\documentclass[12pt,reqno]{amsart}
\usepackage{enumitem,amsmath, amsthm, amssymb, amsfonts, mathtoo
ls,fullpage,ulem} 
\usepackage{amsbsy,amscd,latexsym,amstext,delarray,caption,blkarray,ytableau}
\usepackage{thm-restate}
\usepackage[T1]{fontenc}
\usepackage[utf8]{inputenc} 
\usepackage[dvipsnames]{xcolor}
\usepackage{float}
\usepackage{tikz}
\usetikzlibrary{cd}
\usetikzlibrary{matrix,arrows}
\usepackage{graphicx}
\usepackage{comment}
\usepackage[colorlinks=true,citecolor=black,linkcolor=black,urlcolor=blue]{hyperref}

\usepackage{cleveref}

\allowdisplaybreaks[4]
\newtheorem{theorem}{Theorem}[section]

\newtheorem{lemma}[theorem]{Lemma}
\newtheorem{corollary}[theorem]{Corollary}
\newtheorem{proposition}[theorem]{Proposition}

\theoremstyle{definition}
\newtheorem{definition}{Definition}
\newtheorem{example}{Example}
\newtheorem{notation}{Notation}

\newcommand{\E}{\mathbb{E}}

\theoremstyle{remark}
\newtheorem{remark}{Remark}
\usepackage{ragged2e}
\newcommand{\x}{\textbf{x}}
\newcommand{\North}{\texttt{N}}
\newcommand{\East}{\texttt{E}}

\newcommand{\fd}{\Delta_{\ell} f}
\newcommand{\fdm}{\Delta_{m} f}
\newcommand{\SetL}[1]{S_{\Delta_{#1}}}
\newcommand{\mL}{\mathcal{L}}
\newcommand{\des}{\mathrm{des}}
\newcommand{\asc}{\mathrm{asc}}

\DeclareMathOperator{\PF}{PF}
\DeclareMathOperator{\PPF}{PPF}
\DeclareMathOperator{\disp}{dis} 
\DeclareMathOperator{\Des}{Des}

\DeclareMathOperator{\Asc}{Asc}
\DeclareMathOperator{\Tie}{Tie}

\newcommand{\PR}{\mathbb{P}}
\newcommand{\Luka}{\L ukasiewicz}
\newcommand{\LP}{\mathcal{L}}

\usepackage[colorinlistoftodos]{todonotes}

\newcommand{\tb}[1]{{\bf\color{blue}{#1}}}

\newcommand{\defterm}{\textbf}

\title{
On statistics of prime parking functions,\\ \L ukasiewicz~paths, and quasisymmetric functions
}

\author{Pamela E. Harris}
\address[P.~E. Harris]{Department of Mathematical Sciences, University of Wisconsin-Milwaukee, Milwaukee, WI 53211}
\email{\textcolor{blue}{\href{mailto:peharris@uwm.edu}{peharris@uwm.edu}}}

\author[Kara]{Selvi Kara} 
\address[S.~Kara]{Department of Mathematics, Bryn Mawr College, Bryn Mawr, PA 19010}
\email{\textcolor{blue}{\href{mailto:skara@brynmawr.edu}{skara@brynmawr.edu}}}

\author{Erin McNicholas}
\address[E.~McNicholas]{Department of Mathematics, Willamette University,  Salem, OR 97301}
\email{\textcolor{blue}{\href{mailto:emcnicho@willamette.edu}{emcnicho@willamette.edu}}}

\author{Kathryn Nyman}
\address[K.~Nyman]{Department of Mathematics, Willamette University,  Salem, OR 97301}
\email{\textcolor{blue}{\href{mailto:knyman@willamette.edu}{knyman@willamette.edu}}}

\author{Mei Yin}
\address[M.~Yin]{Department of Mathematics, University of Denver, Denver, CO}
\email{\textcolor{blue}{\href{mailto:mei.yin@du.edu}{mei.yin@du.edu}}}

\begin{document}

\begin{abstract}
We recall that 
a parking function of length $n+1$ is said to be prime 
if removing any instance of 1 yields a parking function of length $n$. 
In this article, we study prime parking functions from multiple lenses. We derive an explicit formula for the
average value of the total displacement of prime parking functions. 
We present a formula for the displacement-enumerator of prime parking functions that involves a sum over \Luka~paths. We describe the one-to-one correspondence between parking functions and labeled \Luka~paths via Dyck paths. We introduce the concept of $\ell$-forward differences and use this as a vehicle for examining ties, ascents, and descents in prime parking functions. We establish a link between Schur functions corresponding to the partition $(i,1^{n-i})$ and fundamental quasisymmetric functions indexed by prime parking function tie sets of size $n-i.$

\end{abstract}



\maketitle

\tableofcontents

\section{Introduction}

Classical parking functions were first introduced by Konheim and Weiss \cite{bib:KW} in the study of the linear probes
of random hashing functions. The original conceptual setting was a linear parking lot with $n$ cars, labeled $1$ through $n$, each with a stated parking preference. In their labeled order, each car
attempts to park in its preferred spot. If the car finds its preferred spot occupied, it
moves to the next available spot. Formally, a \defterm{(classical) parking function} of length $n$ is a sequence $\pi=(\pi_1,\pi_2,\dots, \pi_n)$ 
of positive integers such that if $\lambda_1\leq\lambda_2\leq \cdots\leq \lambda_n$ is the increasing
rearrangement of $\pi_1,\pi_2,\dots,\pi_n$, then $\lambda_i\leq i$ for $1\leq i\leq n$. 
We write $\PF_n$
for the set of classical parking functions of length $n$. It is well-known that $\vert \PF_n \vert=(n+1)^{n-1}$.

Parking functions have connections to many areas of mathematics including to noncrossing trees \cite{BrusselSprouts}, hyperplane arrangements \cite{hyperplanes},  representation theory \cite{Steph} and also to computer science via the study of hashing functions \cite{konheimWeiss} and the QuickSort algorithm \cite{quicksort}.   
There are also many generalizations of parking functions including parking sequences and parking assortments \cite{sequencesassortments} and friendship and cyclic parking functions \cite{friendship}. 
Additionally, many have studied discrete statistics of parking functions including work on descents, inversions, and fixed points \cite{pfstats,inversions,fixedpoints}.
We refer to Yan \cite{yan2015parking} for a comprehensive survey on the subject.

An important subset of parking functions are prime parking functions.
A parking function $\pi$ is said to be a \defterm{prime parking function} of length $n+1$, if removing a 1 from the tuple $\pi$ results in a parking function of length $n$. For example, $(3,2,1,1)$ is a prime parking function of length 4, since removing either instance of the value 1 results in the preference list $(3,2,1)$ which is a parking function of length 3. 
 Notice that a prime parking function of length $n+1$ never has an entry $n+1$ because then removing a 1 would not yield a parking function of length $n$.
We denote the set of prime
parking functions of length $n+1$ by $\PPF_{n+1}$. It is well known that $\vert \PPF_{n+1} \vert=n^{n}$; for a proof we point the reader to the work of Kalikow \cite[pp.~141-142]{Stanley}.

In this article, we study (prime) parking functions from multiple lenses. We start with the average displacement of prime parking functions (\Cref{def:displacement}). The displacement of a parking function records the total number of additional spaces cars must travel beyond their preferred spots in order to park. Before deriving an explicit formula for the average displacement of prime parking functions, we first compute the expected value of a single coordinate, $\pi_1$. Because the set of (prime) parking functions is closed under permutations of the cars, all coordinates have the same marginal distribution. In particular, for $\pi=(\pi_1,\pi_2,\ldots, \pi_n)$ chosen uniformly at random from $\PPF_{n+1}$, we have
$$\E[\pi_i | \pi \in \PPF_{n+1}]= \E [\pi_1 | \pi \in \PPF_{n+1}]$$
for all $i$. Thus, it suffices to compute $\E[\pi_1 | \pi \in \PPF_{n+1}]$. We prove the following exact and asymptotic results, corresponding to the classical results, 
\begin{eqnarray*}
    \E[\pi_i|\pi\in \PF_n]&=&\frac{n+1}{2}-\frac{(n-1)!}{2(n+1)^{n-1}}\sum_{s=0}^{n-2}\frac{(n+1)^s}{s!}\\
    &=&\frac{n+1}{2}-\left[\frac{\sqrt{2\pi}}{4}n^{1/2}-\frac{7}{6}\right]+o(1),
\end{eqnarray*}
given in \cite{durmic2022probabilistic}, Theorem 4.

\begin{restatable}[]{theorem}{expectedvalueresult}
\label{thm:expected value of pi1}
    The expected value of $\pi_1$ for prime parking functions $\pi\in \PPF_{n+1}$ is 
    \[ \E [\pi_1 | \pi \in \PPF_{n+1}]=\frac{1}{2}\left(n+3-\frac{n!}{n^n}\sum_{s=0}^n\frac{n^{s}}{s!}\right).\]
\end{restatable}

\begin{restatable}{corollary}{expectedvalueresultasymp}
\label{cr:expected value of pi1}
 The expected value of $\pi_1$ for prime parking functions $\pi\in \PPF_{n+1}$ is 
    \[ \E [\pi_1 | \pi \in \PPF_{n+1}]=\frac12 \left (n-\sqrt{\frac{\pi n}{2}} + \frac73 + o(1) \right ).\]
\end{restatable}
These prime parking function analogues of the classical results are related, but independent findings.  While prime parking functions are a specialization of parking functions, results about them do not follow directly from the corresponding classical theorems.  Prime parking functions of length $n+1$ can be derived from classical parking functions of length $n$ by reinserting a $1,$ but this transformation is not straightforward.  Many results depend on whether the $1$ is inserted before or after the first entry, and on how many $1$'s there are in total.  As a result, the corresponding proofs for prime parking functions often require more complex arguments and calculations.

Using \Cref{thm:expected value of pi1} and \Cref{cr:expected value of pi1}, we obtain the expected value for the displacement of prime parking functions.

\begin{restatable}{corollary}{expecteddisplacementforPPFs}
\label{expecteddisplacementforPPFs}
The expected value of the displacement of prime parking functions is
\[\E [\text{dis}(\pi)|\pi\in\PPF_{n+1}]=\frac{\sqrt{2\pi}}{4}n^{3/2}-\frac{n}{6}+o(n).\]    
\end{restatable}

We also study the displacement-enumerator of prime parking functions. Just as classical parking functions may be encoded by \Luka~paths, prime parking functions may also be encoded by \Luka~paths, except that the path stays strictly above the $x$-axis other than at its start and end points $(0, 0)$ and $(n, 0)$. 
This alternative interpretation transforms the total displacement of a parking function into an area enclosed by the associated \Luka~path and the $x$-axis. For more on these bijections, see the work of Selig and Zhu \cite{SZ} and references therein. 
We relate the displacement-enumerator of prime parking functions to weights of \Luka~paths, building upon earlier work for classical parking functions in Elvey-Price \cite{Elvey}. Our main result is as follows.

\begin{restatable}{theorem}{PPFDisplacementEnumerator}
\label{thm:displacement-enumerator}
The displacement-enumerator of prime parking functions is
\[
\PPF_{n+1}(q)=(n+1)! ~~q^n \sum_{\substack{\text{all possible}\\\text{\Luka~paths with} \\ \text{height sequence } \\
(h_0, h_1, \dots, h_n)}} \frac{1}{h_1-h_0+2}  ~~\prod_{j=1}^{n} \frac{q^{h_{j}}}{(h_{j}-h_{j-1}+1)!},
\]
where $\frac{q^{h_{j}}}{(h_{j}-h_{j-1}+1)!}$ denotes the weight of a step $h_{j-1} \rightarrow h_{j}$ with $h_{j} \geq h_{j-1}-1$ in the \Luka~path.  
\end{restatable}

We then apply the circular rotation construction for prime parking functions due to Kalikow \cite[pp.~141-142]{Stanley} to study the generating function for ties, ascents, and descents in prime parking functions. We introduce $\ell$-forward differences $\Delta_\ell f(\pi)$ in prime parking functions. This concept generalizes ties, which are $0$-forward differences, and is also related to ascents and descents. Our main result is the following.

\begin{restatable}{proposition}{PPFRotationCount}
\label{prop:q_l-des}
Take any integer $\ell \in [0, n-2]$. Then we have
\begin{equation}
\sum_{\pi \in \PPF_n} q^{
\fd(\pi)}=(q+n-2)^{n-1}.
\end{equation}
\end{restatable}

Lastly, we bring quasisymmetric functions $F_{n, S}$ and Schur functions $s_\lambda$ into the mix and further demonstrate the intriguing characteristics displayed by $\ell$-forward differences in parking functions. 
The second part of the following theorem is a prime analogue of the classical result derived by Celano, \cite{Celano:SchurPositivity}.  The prime and classical results are independent, but related.  In both cases, Kalikow's circular rotation construction of parking functions is used to show the interconnections between parking, quasisymmetric, and Schur functions.  Our results are in the context of $l$-forward differences while Celano's work is on ties and small cyclic descents ($1$-forward differences).

\begin{restatable}{theorem}{PPFQuasiSymmetric}
\label{thm:quasisym-corr}
Let $m \neq \ell$. Then we have
$$ \sum_{\pi \in \PPF_n} q^{
\fd(\pi)} F_{n, \SetL{\ell}(\pi)}=\sum_{i=1}^n q^{n-i}(n-2)^{i-1} s_{(i,1^{n-i})},$$
and
$$ \sum_{\pi \in \PPF_n} q^{
\fd(\pi)} F_{n, \SetL{m}(\pi)}= \sum_{i=1}^n (q+n-3)^{i-1} s_{(i,1^{n-i})}.$$
\end{restatable}

This article is organized as follows. In \Cref{sec:expected value results}, we prove our results related to the expected value of $\pi_1$ for a prime parking function and obtain the average displacement in prime parking functions. In \Cref{sec: Displacement Enum}, we give the necessary background on \Luka~paths and express the displacement-enumerator of prime parking functions as a sum over \Luka~paths, with weights determined by step heights. 
In \Cref{sec:L paths}, we use a labeled version of \Luka~paths and describe the ties, ascents, and descents in parking functions directly from the labeled paths.
In \Cref{sec:l-forward diff}, we present $\ell$-forward differences as another vehicle for examining ties, ascents, and descents in prime parking functions.  We find the generating function for the number of $\ell$-forward differences, and use this function to derive that the expected number of ties in a prime parking function is $1$, while the expected number of descents and the expected number of ascents are both equal to $(n-2)/2$.  In \Cref{sec: TieSets QuasiSymm}, we establish a link between Schur functions corresponding to the partition $(i,1^{n-i})$ and fundamental quasisymmetric functions indexed by prime parking function tie sets of size $n-i.$ 
We conclude with some directions for future investigation.

\section{Expected value results}\label{sec:expected value results}

A principle that plays an important role in our analysis is that, as with classical parking functions, prime parking functions are invariant under permutation. 
That is, if $\pi\in\PPF_n$, and if $\lambda$ is any rearrangement
of the entries in $\pi$, then $\lambda \in \PPF_n$.
By symmetry, the uniform distribution on (prime) parking functions is invariant under permutations of the cars and all coordinates have the same expected value.

Hence, for $1\leq i\leq n+1$,
\begin{align}
    \E[\pi_i | \pi \in \PPF_{n+1}]&=\E[\pi_1 | \pi \in \PPF_{n+1}] \notag \\ 
    &=\sum_{j=1}^{n} j\cdot \PR(\pi_1=j|\pi\in \PPF_{n+1})\label{eq:second equality}\\
    &=\sum_{j=1}^n j\cdot \frac{|
    \{\pi\in\PPF_{n+1}:\pi_1=j\}|}{n^{n}}\label{eq:expectedvalue},
\end{align}
where in 
\Cref{eq:second equality} 
we use the fact that none of the entries in $\pi \in \PPF_{n+1}$ can be $n+1$.

\begin{example} 
Since  $\PPF_3=\{(1,1,1),(1,2,1),(1,1,2),(2,1,1)\}$, we have
    \[\E [\pi_1 | \pi \in \PPF_{3}]=\frac{1\cdot 3+2\cdot 1}{4}=\frac{5}{4}.\]
\end{example}

Write $[n]$ for the set of integers $\{1,2, \dots, n\}$. Subsequently, for any tuple $\alpha\in[n]^n$, we let $N_1(\alpha)$ be the number of ones in $\alpha$. 
Define \[f_{n}(j,k)=|\{\pi\in\PF_n:\pi_1=j, \ N_1(\pi)=k\}|.\]
The count $f_n(j,k)$ is for $\PF_n$, but it will play a key role in deriving the expected value $\E[\pi_1 | \pi \in \PPF_{n+1}]$ given in \Cref{eq:expectedvalue}, as we see in the following theorem.

\begin{theorem}
\label{thm:related to fnjk}
Let $n\geq 1$. 
For $j=1$, we have
\[|\{\pi\in\PPF_{n+1}:\pi_1=j\}|=|\PF_{n}|=(n+1)^{n-1}.\]
For $j=2,3,\ldots, n$, we have 
    \[|\{\pi\in\PPF_{n+1}:\pi_1=j\}|=\sum_{k=1}^{n-1}\frac{n}{k+1}f_{n}(j,k).\]
\end{theorem}
\begin{proof}
For the $j=1$ case,  since $\pi\in\PPF_{n+1}$ satisfies $\pi_1=1$, removing this instance of 1 returns a parking function of length $n$.  There are $(n+1)^{n-1}$ parking functions of length $n$, and so this case follows.

 For the $j>1$ case, let $\sigma \in \PF_n$ with $N_1(\sigma)=k$, i.e., $\sigma$ has $k$ ones, and such that $\sigma_1=j$. Since $\sigma_1=j>1$, we have $1\leq k\leq n-1$. We can form a parking function $\pi \in \PPF_{n+1}$ with $k+1$ ones and $\pi_1=j$ by inserting a 1 in one of the $n$ spots between elements of $\sigma$ or at the end of $\sigma$. Each $\pi \in \PPF_{n+1}$ with $\pi_1=j$ will arise by this construction $k+1$ times; once for each $i$ with $\pi_i=1$. 
\end{proof}

A central ingredient in \Cref{thm:related to fnjk} is $f_n(j, k)$. We develop a formula for its value when $j=2,3,\ldots,n$ and $k=1,2,\ldots,n-1$. To this end, we recall the definition of a parking function shuffle. 

\begin{definition}[Parking Function Shuffle, p.~129 \cite{diaconis2017probabilizing}]
We say that $\pi_2,\pi_3,\ldots,\pi_n$ is a \defterm{parking function shuffle} of $\alpha\in\PF_{k-1}$ and $\beta\in \PF_{n-k}$, if $\pi_2,\pi_3,\ldots,\pi_n$ is any permutation of the union of the two words $\alpha$ and $\beta+(k)$, where $\beta+(k)$ is the set of values of $\beta$ shifted by $k$. The set of all such shuffles is denoted by $Sh(k-1,n-k)$.
\end{definition}

For example, a shuffle of $\alpha=(1,3,2,2)$ and $\beta =(2,1,2)$ is 
$(3,\underline{6},2, 1, 2,\underline{7},\underline{7})$ where we underline the original $\beta$ shifted by $k=5$. 
Diaconis and Hicks \cite{diaconis2017probabilizing} proved that $(k, \pi_2,\pi_3,\ldots,\pi_n)$ is a valid parking function if and only if $(\pi_2,\pi_3,\ldots,\pi_n) \in Sh(\ell-1,n-\ell)$ for some $\ell\geq k$, leading to the following count.

\begin{corollary}[Corollary 1 in \cite{diaconis2017probabilizing}]\label{cor:PF with pi_1=j}
    The number of $\pi\in\PF_n$ with $\pi_1=j$ is
    \[\sum_{\ell=j}^{n}\binom{n-1}{\ell-1}|\PF_{\ell-1}|\cdot|\PF_{n-\ell}|=\sum_{\ell=j}^{n}\binom{n-1}{\ell-1}\ell^{\ell-2}\cdot(n-\ell+1)^{n-\ell-1}.\]
\end{corollary}

The following result is well-known. 

\begin{theorem}[Corollary 1.16 in \cite{yan2015parking}]\label{thm:PF with k ones}
The number of parking functions $\pi\in \PF_{n}$ with exactly $k$ ones, i.e. $N_1(\pi)=k$, is given by 
\[\bigl|\{\pi \in \PF_n: N_1(\pi)=k\}\bigr| = \binom{n-1}{k-1}n^{n-k}.\]
\end{theorem}

We are now ready to give a count for $f_n(j, k)$, the number of parking functions of length $n$ with $k$ ones and fixed initial value $j$.

\begin{proposition}\label{prop:f_n(j,k)_count}
    The number of $\pi\in\PF_n$ with $\pi_1=j$ and $N_1(\pi)=k$ is given by
    \begin{equation}\label{eq:f_n(j,k)}
        f_n(j,k)=\sum_{\ell=j}^n\binom{n-1}{\ell-1} \binom{\ell-2}{k-1}(\ell-1)^{\ell-1-k} \cdot (n-\ell+1)^{n-\ell-1}.
    \end{equation}
    
\end{proposition}
\begin{proof}
Now when shuffling $\alpha\in\PF_{\ell-1}$ and $\beta\in\PF_{n-\ell}$, the only ones appearing in the shuffle must have come from $\alpha$, which we know has length $\ell-1$. This is because when shuffling $\alpha$ and $\beta$ we would increase all values in $\beta$ by $\ell$. This means that we have to count all possible $\alpha\in\PF_{\ell-1}$ having  $k$ ones, which by \Cref{thm:PF with k ones} is given by 
\[\binom{\ell-2}{k-1}(\ell-1)^{\ell-1-k}.\]
Replacing the factor of $|\PF_{\ell-1}|=\ell^{\ell-2}$ with $\binom{\ell-2}{k-1}(\ell-1)^{\ell-1-k}$ in \Cref{cor:PF with pi_1=j} we get that the number of $\pi\in\PF_{n}$ with $\pi_1=j$ and $N_1(\pi)=k$ is
    \begin{align*}
        &\sum_{\ell=j}^n\binom{n-1}{\ell-1} \binom{\ell-2}{k-1}(\ell-1)^{\ell-1-k} \cdot |\PF_{n-\ell}|\\
        &\hspace{1in}=\sum_{\ell=j}^n\binom{n-1}{\ell-1} \binom{\ell-2}{k-1}(\ell-1)^{\ell-1-k} \cdot (n-\ell+1)^{n-\ell-1}.\qedhere
    \end{align*}
\end{proof}

\begin{remark}\label{rk:pos}
    Since the binomial coefficient $\binom{n}{k}$ is $0$ for $k>n,$ the sum in \eqref{eq:f_n(j,k)} is equivalent to \[f_n(j,k)=\sum_{\ell=\max(j,k+1)}^n\binom{n-1}{\ell-1} \binom{\ell-2}{k-1}(\ell-1)^{\ell-1-k} \cdot (n-\ell+1)^{n-\ell-1}.\] 
\end{remark}

We now come to our main result. The proof makes use of special cases of Abel's extensions of the binomial theorem, a portion of which we recall below.

\begin{theorem}[Abel's extension of the binomial theorem, derived from Pitman \cite{Pitman} and Riordan \cite{Riordan}]\label{Abel}
Let
\begin{equation*}\label{b}
A_n(x, y; p, q)=\sum_{s=0}^n \binom{n}{s} (x+s)^{s+p} (y+n-s)^{n-s+q}.
\end{equation*}
Then
\begin{equation*}\label{b2}
A_n(x, y; p, q)=A_{n-1}(x, y+1; p, q+1)+A_{n-1}(x+1, y; p+1, q), \textrm{ and}
\end{equation*}
\begin{equation*}\label{b3}
A_n(x, y; p, q)=\sum_{s=0}^{n} \binom{n}{s}s!(x+s)A_{n-s}(x+s, y; p-1, q).
\end{equation*}
Moreover, the following special instances hold via the basic recurrences listed above:
\begin{equation}\label{2}
A_n(x, y; -1, 0)=x^{-1}(x+y+n)^n.
\end{equation}
\begin{equation}\label{3}
A_n(x, y; -1, 1)=x^{-1} \sum_{s=0}^n \binom{n}{s} (x+y+n)^s (y+n-s) (n-s)!.
\end{equation}
\end{theorem}

\expectedvalueresult*

\begin{proof}
Using \Cref{thm:related to fnjk} we have
    \begin{align} 
    \E[\pi_1 | \pi \in \PPF_{n+1}]&=\sum_{j=1}^{n}j\cdot \PR(\pi_1=j|\pi\in \PPF_{n+1})\nonumber\\
    &=\frac{1}{n^n}\left((n+1)^{n-1}+\sum_{j=2}^n j\sum_{k=1}^{n-1}\frac{n}{k+1}f_n(j,k)\right).\label{eq:expected value}
    \end{align}
   For ease of notation, denote the following sum by $A$:   
    \begin{align}
      A&\coloneqq \sum_{j=2}^n j\sum_{k=1}^{n-1}\frac{n}{k+1}f_n(j,k)\nonumber\\
      &=  \sum_{j=2}^n j\sum_{k=1}^{n-1}\frac{n}{k+1}\sum_{\ell=j}^n\binom{n-1}{\ell-1}\binom{\ell-2}{k-1}(\ell-1)^{\ell-1-k}(n-\ell+1)^{n-\ell-1},\label{eq:A}
    \end{align}
    where we apply Proposition \ref{prop:f_n(j,k)_count} in the second equality. Notice that as pointed out in \Cref{rk:pos}, the binomial coefficient \(\binom{\ell-2}{k-1}\) is non-zero only when $k \leq \ell-1$ ($k$ ones in a parking function of length $\ell-1$). 
    In \eqref{eq:A}, making the change of variables $s=n-\ell$, we find
    \begin{align}
         A &=\sum_{j=2}^{n}j\sum_{k=1}^{n-1}\frac{n}{k+1}\sum_{s=0}^{n-j}\binom{n-1}{s}\binom{n-s-2}{k-1}(n-s-1)^{n-s-k-1}(s+1)^{s-1}\nonumber\\
        &=\sum_{k=1}^{n-1}\frac{n}{k+1} \Big( \sum_{j=2}^{n}j \sum_{s=0}^{n-j}\binom{n-1}{s}\binom{n-s-2}{k-1}(n-s-1)^{n-s-k-1}(s+1)^{s-1} \Big)\nonumber\\
        &=\sum_{k=1}^{n-1}\frac{n}{k+1} \Bigg( \sum_{s=0}^{n-2}  \binom{n-1}{s}\binom{n-s-2}{k-1}(n-s-1)^{n-s-k-1}(s+1)^{s-1}  \Big(\sum_{j=2}^{n-s} j\Big)\Bigg), \label{eq:A'}
    \end{align}
 where the last equality in \eqref{eq:A'} is obtained by interchanging the order of summation. Replacing the sum of $j$ with $\frac{1}{2} (n-s-1)(n-s+2)$, we can reorganize $A$ as follows:
\begin{align}
A&=\frac{1}{2}\sum_{k=1}^{n-1}\frac{n}{k+1} \Bigg( \sum_{s=0}^{n-2}  \binom{n-1}{s}\binom{n-s-2}{k-1}(n-s-1)^{n-s-k}(s+1)^{s-1}   (n-s+2)\Bigg)\nonumber\\
&=\frac{1}{2}\sum_{s=0}^{n-1} \binom{n-1}{s} (s+1)^{s-1}   (n-s+2) \Bigg( \sum_{k=1}^{n-s-1}  \binom{n-s-2}{k-1}(n-s-1)^{n-s-k} \frac{n}{k+1}\Bigg),\label{eq:will simplify this junk}
\end{align}
where the last equality in \eqref{eq:will simplify this junk} involves interchanging the order of summation and setting a more explicit upper bound on $k$ for the binomial coefficient $\binom{n-s-2}{k-1}$ to be non-zero. We also note that the inner sum for $k$ becomes empty when taking $s=n-1$. Thus, changing the upper bound on $s$ from $n-2$ to $n-1$ has no effect on the value of $A$; we are implementing it for the application of Abel's binomial theorem later. We proceed to simplify the inner sum over the indexing variable $k$ to arrive at
\begin{align}
    &\sum_{k=1}^{n-s-1}\frac{n}{k+1}\binom{n-s-2}{k-1}(n-s-1)^{n-s-k}\nonumber\\
   &=\frac{n}{n-s}\sum_{k=2}^{n-s}(k-1)\binom{n-s}{k}(n-s-1)^{n-s-k}\qquad\mbox{(by reindexing $k \leftarrow k+1$)}\nonumber\\
    &=\frac{n}{n-s}\Bigg[\sum_{k=0}^{n-s}(k-1)\binom{n-s}{k}(n-s-1)^{n-s-k}\nonumber\\
    &\hspace{1in}-\left(-1\binom{n-s}{0}(n-s-1)^{n-s}+0\binom{n-s}{1}(n-s-1)^{n-s-1}\right)\Bigg]\nonumber\\
    &=\frac{n}{n-s}\Bigg[\sum_{k=0}^{n-s}k\binom{n-s}{k}(n-s-1)^{n-s-k}\nonumber\\
    &\hspace{1in}-\sum_{k=0}^{n-s}\binom{n-s}{k}(n-s-1)^{n-s-k}+(n-s-1)^{n-s}\Bigg].\label{return here}
\end{align}
We now simplify the first sum appearing in \eqref{return here} to arrive at
\begin{align}
    &\sum_{k=0}^{n-s}k\binom{n-s}{k}(n-s-1)^{n-s-k}\nonumber\\
    &=\sum_{k=1}^{n-s}k\frac{(n-s)!}{k!(n-s-k)!}(n-s-1)^{n-s-k}\qquad\mbox{ (since the $k=0$ term is zero)}\nonumber\\
    &=\sum_{k=1}^{n-s}\frac{(n-s)(n-s-1)!}{(k-1)!(n-s-k)!}(n-s-1)^{n-s-k}\nonumber\\
    &=\sum_{k=1}^{n-s}(n-s)\binom{n-s-1}{k-1}(n-s-1)^{n-s-k}\nonumber\\
    &=\sum_{k=0}^{n-s-1}(n-s)\binom{n-s-1}{k}(n-s-1)^{n-s-k-1}\qquad\mbox{(by reindexing $k\leftarrow k+1$)}\nonumber\\
    &=(n-s)^{n-s}\qquad\mbox{(by the binomial theorem)}\label{messy}\nonumber
\end{align}
The second sum in \eqref{return here} simplifies, by the binomial theorem, to 
\begin{align}
    \sum_{k=0}^{n-s}\binom{n-s}{k}(n-s-1)^{n-s-k}=(1+(n-s-1)^{n-s})
    &=(n-s)^{n-s}.
\nonumber\end{align}
Thus the entire expression 
\eqref{return here} reduces to
\[\frac{n}{n-s}(n-s-1)^{n-s}.\]
Returning to \eqref{eq:will simplify this junk} we have that 
\begin{align}
A&=\frac{1}{2}\sum_{s=0}^{n-1}\binom{n-1}{s}(s+1)^{s-1}(n-s+2)\frac{n}{n-s}(n-s-1)^{n-s} \nonumber\\
    &=\frac{1}{2}\sum_{s=0}^{n-1}\binom{n}{s}(n-s+2)(n-s-1)^{n-s}(s+1)^{s-1}\nonumber\\
    &=\frac{1}{2}\sum_{s=0}^{n-1}\binom{n}{s}((n-s-1)+3)(n-s-1)^{n-s}(s+1)^{s-1}\nonumber\\
    &=\frac{1}{2}\left(\sum_{s=0}^{n}\binom{n}{s}((n-s-1)+3)(n-s-1)^{n-s}(s+1)^{s-1}\right)-\frac{1}{2}2(-1)^0(n+1)^{n-1}\nonumber\\
    &=\frac{1}{2}\left(\sum_{s=0}^{n}\binom{n}{s}(n-s-1)^{n-s+1}(s+1)^{s-1}+\sum_{s=0}^{n}3\binom{n}{s}(n-s-1)^{n-s}(s+1)^{s-1}\right)
    -(n+1)^{n-1}\nonumber\\
    &=
    \frac{1}{2}\left(I+II\right)-(n+1)^{n-1},\label{eq:so close}
\end{align}
where 
\begin{align}
    I&=\sum_{s=0}^{n}\binom{n}{s}(n-s-1)^{n-s+1}(s+1)^{s-1}\label{eq:I}\\
    \intertext{and}
    II&=\sum_{s=0}^{n}3\binom{n}{s}(n-s-1)^{n-s}(s+1)^{s-1}\label{eq:II}.
\end{align}
By Abel's \eqref{3} and \eqref{2}, we have that \eqref{eq:I}  and \eqref{eq:II}, respectively become
\begin{align}
    I&=A_n(1,-1;-1,1)\nonumber\\
    &=\sum_{s=0}^{n}\binom{n}{s}n^s(n-s-1)(n-s)!\nonumber\\
    &=\sum_{s=0}^n\frac{n!}{s!}n^s(n-1)-\sum_{s=0}^n\frac{n!}{s!}n^ss\nonumber\\
    &=\sum_{s=0}^n\frac{n!}{s!}n^{s}(n-1)-\sum_{s=0}^{n-1}\frac{n!}{s!}n^{s+1}\nonumber\\
    &=\sum_{s=0}^{n-1}\frac{n!}{s!}n^{s+1}+\frac{n!}{n!}n^{n+1}-\sum_{s=0}^n\frac{n!}{s!}n^{s}-\sum_{s=0}^{n-1}\frac{n!}{s!}n^{s+1}\nonumber\\
    &=n^{n+1}-\sum_{s=0}^n\frac{n!}{s!}n^{s}
    \label{eq:new I}\\
    \intertext{and}
II&=3A_n(1,-1;-1,0)=3n^n\label{eq:new II}.
\end{align}

Finally, substituting \eqref{eq:new I} and \eqref{eq:new II} into \eqref{eq:so close} we have that 
\begin{align}
    A&=\frac{1}{2}\left(n^{n+1}-\sum_{s=0}^n\frac{n!}{s!}n^{s}+3n^n\right)-(n+1)^{n-1}.\label{eq:yay}
\end{align}
Substituting \eqref{eq:yay} into \eqref{eq:expected value} we find the expected value
\begin{align*}
   \E[\pi_1 | \pi \in \PPF_{n+1}]
    &=\frac{1}{n^n}\left((n+1)^{n-1}+\frac{1}{2}\left(n^{n+1}-\sum_{s=0}^n\frac{n!}{s!}n^{s}+3n^n\right)-(n+1)^{n-1}\right)\\
    &=\frac{1}{2n^n}\left(n^{n+1}-\sum_{s=0}^n\frac{n!}{s!}n^{s}+3n^n\right)\\
    &=\frac{1}{2}\left(n+3-\frac{n!}{n^n}\sum_{s=0}^n\frac{n^{s}}{s!}\right),
\end{align*}
which completes the proof.
\end{proof}

Building upon \Cref{thm:expected value of pi1}, we derive the asymptotics for the expected value of $\pi_1$ for large $n$. First, we need a technical lemma.
\begin{lemma} \cite[Lemma 21]{durmic2022probabilistic}\label{CLT}
 Let $X_1, X_2, \dots$ be iid Poisson$(1)$ random variables. Then
\begin{equation*}
\PR(X_1+\cdots+X_n \leq n)=\frac{1}{2}+\frac{2}{3}\frac{1}{\sqrt{2\pi n}}+o\left(\frac{1}{\sqrt{n}}\right).
\end{equation*}
\end{lemma}

\expectedvalueresultasymp*

\begin{proof}
We recognize that
\begin{equation*}
e^{-n} \sum_{s=0}^{n} \frac{n^s}{s!}
\end{equation*}
in \Cref{thm:expected value of pi1} equals the probability that the sum of $n$ iid Poisson$(1)$ random variables is less than or equal to $n$, and so its asymptotics may be estimated by \Cref{CLT}:
\begin{align}
e^{-n} \sum_{s=0}^n \frac{n^s}{s!} = \frac{1}{2}+\frac{2}{3}\frac{1}{\sqrt{2\pi n}}+o\left(\frac{1}{\sqrt{n}}\right).\label{eq:first}
\end{align}
By Stirling's approximation:
\begin{align} n!= \sqrt{2\pi n} e^{-n} n^n \left(1+\frac{1}{12n}+ o\Big(\frac{1}{n}\Big)\right).\label{eq:second}
\end{align}
Using \eqref{eq:first} and \eqref{eq:second}, we can express the expected value derived in \Cref{thm:expected value of pi1} as follows: 
\begin{align}
\E[ \pi_1 | \pi \in \PPF_{n+1}] &= \frac12 \left ( n+3 -\frac{\sqrt{2\pi n} e^{-n}n^n \left ( 1 + \frac {1}{12n}+ o\left (\frac{1}{n} \right ) \right )}{n^n} \cdot e^n \left (\frac12+\frac23 \frac{1}{\sqrt{2\pi n}}+ o\left (\frac{1}{\sqrt{n}}\right )
\right )\right ) \nonumber
\\
 &= \frac12 \left (n+3 -\sqrt{2\pi n} \left ( \frac12 + \frac23\frac{1}{\sqrt{2\pi n}} + o\left (\frac{1}{\sqrt{n}} \right ) \right ) \right )  \nonumber\\
& = \frac12 \left (n+3 -\sqrt{\frac{\pi n}{2}} - \frac23 + o(1) \right ) \nonumber\\
&= \frac12 \left (n-\sqrt{\frac{\pi n}{2}} + \frac73 + o(1) \right ), \nonumber 
\end{align}
which is our claimed result.
\end{proof}

In the classical parking functions setting,
\cite[Theorem 4]{durmic2022probabilistic} states that for $n$ large and a preference vector $\pi\in[n]^n$ chosen uniformly at random,
\[\E[\pi_1|\pi\in\PF_n]=\frac{n+1}{2}-\frac{\sqrt{2\pi}}{4}n^{1/2}+\frac76+o(1).\]
Using the fact that $(n-1)^{1/2}=n^{1/2}+o(1),$ and applying \Cref{cr:expected value of pi1} to prime parking functions of length $n,$ we find that \[\E[\pi_1|\pi\in \PPF_n]=\frac{n-1}{2}-\frac{\sqrt{2\pi}}{4}(n^{1/2}+o(1))+\frac{7}{6}+o(1).\]  
Thus, asymptotically $\E[\pi_1|\pi\in \PPF_n]$ is one less than $\E[\pi_1|\pi\in \PF_n]$. This accords with our intuition that asymptotically parking preferences would shift down by $1$. An alternative criterion for a classical parking function $\pi=(\pi_1, \pi_2,\dots, \pi_n)$ to be prime is that for all $1\leq i\leq n-1$, at least $i+1$ cars want to park in the first $i$ places, in contrast to the classical requirement that at least $i$ cars want to park in the first $i$ places. Raising the ``congestion'' requirement by one is consistent with each preference value shifting down by one on average.

\subsection{Average displacement of prime parking functions} Now that we have the average value of the first car's preference across all prime parking functions of length $n+1$, we compute the average displacement of prime parking functions.

\begin{definition}\label{def:displacement}
    Given a parking function $\pi=(\pi_1,\pi_2,\ldots,\pi_n)$ in which car $i$ parks in spot $p_i$, 
we say the displacement of car $i$ is $d(i)=p_i-\pi_i$, and the displacement of $\pi$ is given by 
$$\text{dis} (\pi)=\sum_{i=1}^n \left(p_i-\pi_i\right)=\frac{n(n+1)}{2}
-\sum_{i=1}^n\pi_i.$$ 
\end{definition}

\expecteddisplacementforPPFs*

\begin{proof}
Recall that prime parking functions are permutation invariant. Using \Cref{thm:expected value of pi1} and \Cref{cr:expected value of pi1} we have that
    \begin{align*}
    \E [\text{dis} (\pi) | \pi \in \PPF_{n+1}]&=\sum_{i=1}^{n+1}(i-\E[\pi_i| \pi \in \PPF_{n+1}])\\
    &=\frac{(n+1)(n+2)}{2}-(n+1)\E[\pi_1| \pi \in \PPF_{n+1}] \nonumber\\
    &=\frac{(n+1)(n+2)}{2}-\frac{n+1}{2}\left(n-\sqrt{\frac{\pi n}{2}}+\frac73+o(1)
    \right)\nonumber\\
    &=\frac{n+1}{2}\left(2+\sqrt{\frac{\pi n}{2}}-\frac73+
    o(1)\right)\nonumber\\
    &=\frac{\sqrt{2\pi}}{4}n^{3/2}-\frac{n}{6}+o(n). \nonumber\qedhere
    \end{align*}
\end{proof}

\section{Displacement-enumerator results}\label{sec: Displacement Enum}

In their work on the connections between parking functions and lattice paths, Selig and Zhu \cite{SZ} showed that in addition to the well-known connection between parking functions and Dyck paths, there is an intriguing and less studied connection between parking functions and \Luka~paths. In particular, the displacement of a parking function can be interpreted as the area under an associated \Luka~path. This viewpoint gives a helpful way to understand displacement statistics in terms of lattice path geometry.  Building on this idea, we focus on prime parking functions and examine a weighted generating function that keeps track of their total displacement. 
We show that the displacement-enumerator can be expressed as a sum over \Luka~paths, with weights determined by step heights. This result provides a prime analogue of known enumerators for classical parking functions, where the additional constraint that paths stay strictly above the \(x\)-axis (except at the endpoints) adds an interesting layer of structure.

We start by introducing relevant notation and some of the needed terminology from \cite{SZ} relating to \Luka~paths.

\begin{notation}
For the remainder of the paper, we denote the displacement-enumerator of the set of prime parking functions of length $n+1$ by $\PPF_{n+1}(q)$. So, we have
 \[\PPF_{n+1}(q)= \sum_{\pi \in \PPF_{n+1}} q^{\disp(\pi)}.\]
\end{notation}

Next, we recall the definition of a \Luka~word and a \Luka~path, as given by Selig and Zhu in  \cite[Definition 1.3]{SZ}.

\begin{definition}\label{def:luk_word}
A \defterm{\Luka~{word}} of length $n$ is a sequence $\ell = (\ell_1,\ell_2, \ldots, \ell_n)$ of integers $\ell_i \geq -1$, for all $i\in[n]$, such that:
\begin{itemize}[topsep=2pt, noitemsep]
\item For any $k \in [n]$, we have $\sum\limits_{i=1}^k \ell_i \geq 0$.
\item We have $\sum\limits_{i=1}^n \ell_i = 0$.
\end{itemize}
\end{definition}

Following the conventions in \cite{SZ}, 
 \Luka~words can be represented with  lattice paths by associating each $\ell_i$ with a step of the form $(\ell_i+1,\ell_i)$. Throughout this section, we may refer to step  $(\ell_i+1,\ell_i)$ as the step $\ell_i$. 
These  lattice paths are from $(0,0)$ to $(n,0)$ and they never fall below the $x$-axis. We call such a lattice path  a \defterm{\Luka~path} of length $n$ and  denote the set of all \Luka~paths of length $n$ by $\LP_n$.

Given a parking function $\pi=(\pi_1,\pi_2,\ldots, \pi_n)\in \PF_n$, one can associate it to a \Luka~path by defining $\ell=(\ell_1,\ell_2,\ldots,\ell_n)$ where $\ell_j= | \{ i : \pi_i=j\}|-1$ for each $j\in [n]$.  A bijection between the set of 
weakly increasing parking functions and 
\Luka~paths was established in \cite{SZ}. In addition, it was shown in \cite[Theorem 2.1]{SZ} that the area under the \Luka~path is equal to the total displacement of the parking function.

Our goal is to use this connection to find an expression for the displacement-enumerator for prime parking functions. For this purpose, we recall the corresponding notion of a prime \Luka~path \cite{SZ}.

\begin{definition}
    A \Luka~path of length $n$ 
    is called \defterm{prime} if it only touches the $x$-axis at $(0,0)$ and $(n,0)$.
\end{definition}

\begin{example} 
\label{ex:L-path}
Consider the (prime) parking function $\pi=(1,1,1,3,4,4,6)$. The corresponding \Luka~ word is  $\ell=(2,-1,0,1,-1,0,-1)$ and the corresponding  \Luka~path is displayed in blue 
in \Cref{fig:Lukapath}.  As the path only touches the $x$-axis
 at $(0,0)$ and $(7,0)$, it is a prime \Luka~path of length $7$. 
 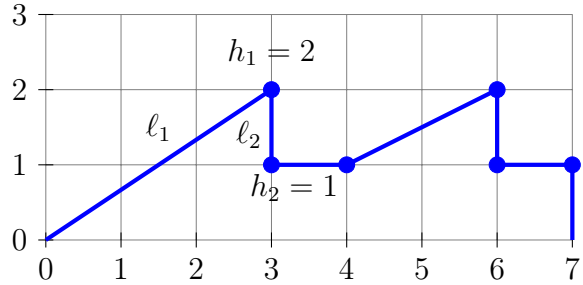
\begin{figure}[H]
    \centering
    \begin{tikzpicture}
           \draw[step=1cm,gray,very thin] (0,0) grid (7,3);
           \foreach \x in {0,...,3} {
        \draw (0.1,\x) -- (-0.1,\x) node[left] {\x};
    }
    \foreach \x in {0,...,7} {
    \draw (\x,0.1) -- (\x,-0.1) node[below] {\x};
    }
    \draw[ultra thick, blue](0,0)--(3,2)--(3,1)--(4,1)--(6,2)--(6,1)--(7,1)--(7,0);
    \foreach \x/\y in { 3/2, 3/1, 4/1,6/2, 6/1, 7/1} {
        \filldraw[blue] (\x,\y) circle (3pt);
    }
    \node at (3,2.5) {$h_1=2$};
    \node at (1.5,1.5) {$\ell_1$};
    \node at (2.7,1.4) {$\ell_2$};
    \node at (3.3,.7) {$h_2=1$};
    \end{tikzpicture}
    \caption{The (prime) \Luka~path corresponding to $\pi = (1,1,1,3,4,4,6) \in \PPF_7$.}
    \label{fig:Lukapath}
\end{figure}
\end{example}

As we see below, one can collect the height information of a \Luka~path (after each step) in a sequence.

\begin{notation}
    Given a parking function $\pi \in \PF_n$ with its associated \Luka~path $\ell$,  define the \textit{height sequence}  of $\pi$ as $h=(h_0,h_1,h_2,\ldots, h_n)$ such that $ h_0= h_n=0$ and
    \[h_j\coloneqq \sum_{i=1}^j \ell_i\] 
    for each $j\in [n-1]$. Notice that  $h_j$  is indeed the height of the corresponding \Luka~path after step $\ell_j$ and $h_{j}-h_{j-1}=\ell_j$ for each $j\in [n]$. 
\end{notation}

\begin{example}
     The height sequence of the (prime) parking function given in \Cref{ex:L-path}  is $h=(0,2,1,1,2,1,1,0)$.
\end{example}

\begin{remark}\label{rem:height_relations}
  Given a parking function $\pi=(\pi_1,\pi_2,\ldots, \pi_n) \in \PF_n$ with the corresponding height sequence $h=(h_0,h_1,\ldots, h_n)$, it is worth noticing that  
    \[h_{j}-h_{j-1}+ 1= |\{ i: \pi_{i}=j\}|\]
   for each $j\in [n]$. As a result, one can express the sum of all the preferences in $\pi$ in the following way:
   \[ \sum_{i=1}^{n}\pi_i = \sum_{j=1}^{n} j(h_j-h_{j-1}+1).\]
Lastly, we have $\displaystyle \sum_{j=1}^{n} (h_{j}-h_{j-1}+1) = \sum_{j=1}^{n} (\ell_{j} +1)= n$.
\end{remark}

In what follows, we use these height sequences in expressing the displacement-enumerator of prime parking functions.

\PPFDisplacementEnumerator*
\begin{proof}
Recall from the definition of prime parking functions that every prime parking function $\pi' \in \PPF_{n+1}$ can be obtained by inserting a $1$ into a parking function $\pi \in \PF_n$. Equivalently, given a prime parking function $\pi'$, we can recover an associated parking function $\pi$ by removing one occurrence of $1$ from $\pi'$. This relationship allows us to connect the set of prime parking functions of size $n+1$ to the set of parking functions of size $n$. See the proof of Theorem \ref{thm:related to fnjk} for details. In particular, we can use this correspondence to express the displacement-enumerator of prime parking functions in terms of classical parking functions, as follows:
\begin{align}
    \PPF_{n+1}(q)=\sum_{\pi' \in \PPF_{n+1}} q^{\disp(\pi')} = \sum_{\pi \in \PF_n} \Big( \frac{n+1}{h_1-h_0+2}  \Big) q^{\sum\limits_{i=1}^{n+1}i-(1+\sum\limits_{i=1}^{n}\pi_i)},\label{eq:above}
\end{align}
where \((h_0, h_1, \ldots, h_n)\) denotes the height sequence associated to the parking function \(\pi \in \PF_n\). Moreover, $h_j-h_{j-1}+1$ is the number of occurrences of $j$ in $\pi$. So,   \(h_1 - h_0 + 2\) is the number of entries equal to 1 in the prime parking function \(\pi' \in \PPF_{n+1}\).  The factor
$\frac{1}{h_1-h_0+2},$ in formula \eqref{eq:above}, accounts for overcounting in the sum over $\pi\in\PF_n$, since each $\pi'$ is counted once for each of its $h_1-h_0+2$ occurrences of $1$.

We focus on the exponent of \( q \) in the final expression before proceeding further. We show that this exponent can be rewritten in terms of the entries of the height sequence. Observe the following sequence of equalities:
$$
q^{\sum\limits_{i=1}^{n+1}i-(1+\sum\limits_{i=1}^{n}\pi_i)}
= q^n q^{(1+2+\cdots+ n) - \sum\limits_{i=1}^{n} i(h_i-h_{i-1}+1)}  = q^n q^{-\sum\limits_{i=1}^{n} i(h_i-h_{i-1})}  = q^n q^{\sum\limits_{i=0}^{n}h_i} = q^n \prod_{j=1}^{n}q^{h_{j}}.
$$
Thus, we have
\[
\PPF_{n+1}(q)= 
\sum_{\pi \in \PF_n} (n+1) q^n\Big( \frac{1}{h_1-h_0+2}  \Big) \prod_{j=1}^{n} q^{h_{j}}.
\]

To obtain the desired expression, we first note that different parking functions can share the same height sequence.  For a fixed height sequence \( (h_0, h_1, \ldots, h_n) \), the corresponding parking function $\pi=(\pi_1,\pi_2, \dots, \pi_n)$ satisfies $h_{j}-h_{j-1}+ 1= |\{ i: \pi_{i}=j\}|$. Summing over height sequences instead of parking functions thus introduces a combined multiplicity factor of $\prod_{j=1}^n \frac{n! }{(h_{j}-h_{j-1}+1)!}$ in the product. Since there is a one-to-one correspondence between \L{}ukasiewicz paths and height sequences, we can sum over all \L{}ukasiewicz paths of length \( n \), as shown below:
\begin{align*}
\PPF_{n+1}(q)
&= (n+1)! ~~q^n \sum_{\substack{\text{all possible}\\\text{\Luka~paths with} \\ \text{height sequence } \\
(h_0, h_1, \dots, h_n)}} \Big( \frac{1}{h_1-h_0+2}  \Big) ~~ \prod_{j=1}^{n} \frac{q^{h_{j}}}{(h_{j}-h_{j-1}+1)!}.  \qedhere
\end{align*}
\end{proof}

\begin{remark}
    In a similar fashion, one can express the displacement-enumerator of prime parking functions  in terms of  prime \Luka~paths as follows: 
\begin{equation*}
\sum_{\pi' \in \PPF_{n+1}} q^{\disp(\pi')} =(n+1)! \sum_{
\substack{
\text{all possible}
\\
\text{\Luka~paths} 
\\
\text{with height sequence} 
\\
(h_0, h_1, \dots, h_{n+1}): \, h_i \geq 1 \, \forall\, 1\leq i\leq n}} \prod_{j=1}^{n+1} \frac{q^{h_{j}}}{(h_{j}-h_{j-1}+1)!}.
\end{equation*}
The corresponding formula for the displacement-enumerator of (classical) parking functions  in terms of \Luka~paths was derived earlier by Elvey-Price in \cite[Lemma 2.3]{Elvey}.
\end{remark}

This concludes our discussion on the displacement-enumeration of prime parking functions in terms of \L{}ukasiewicz paths. In the following section, we turn our attention to the structure of \L{}ukasiewicz paths themselves in order to study additional statistics of prime parking functions.

\section{\Luka~paths and descents}\label{sec:L paths}

In this section, we focus on descent statistics of \Luka~paths.  Descent sets and related statistics on parking functions have been studied in \cite{inversions, harris2024parkingfunctionsfixedset, Schumacher, pfstats}. Our goal is to express the descent statistic in terms of \Luka~paths. We do this by using the labeled version of \Luka~paths, introduced in \cite{SZ},
which naturally correspond to parking functions. This correspondence allows us to describe the descent, ascent, and tie sets directly from the labeled paths. 

We start this section by recalling the notion of descent and other related statistics. Given a tuple $\x = (x_1, x_2, \ldots, x_n) \in [n]^n$,  
we say that $\x$ has a \defterm{descent} at $i$ if $x_i > x_{i+1}$,  
an \defterm{ascent} at $i$ if $x_i < x_{i+1}$,  
and a \defterm{tie} at $i$ if $x_i = x_{i+1}$.  
The \defterm{descent set}, \defterm{ascent set}, and \defterm{tie set} of $\x$ are defined as
\begin{align*}
    \Des(\x) &\coloneqq \{\, i \in [n-1] : x_i > x_{i+1} \,\}, \\
    \Asc(\x) &\coloneqq \{\, i \in [n-1] : x_i < x_{i+1} \,\}, \\
    \Tie(\x) &\coloneqq \{\, i \in [n-1] : x_i = x_{i+1} \,\}.
\end{align*}

We use lower-case versions to denote the cardinality of the corresponding set, e.g. $\des(\x)=|\Des(\x)|.$

\begin{definition}\label{def:labeled_Luka}
    Let $\ell = (\ell_1, \ell_2, \ldots, \ell_n)$ be a \Luka~path of length $n$,  
    and let $\mathcal{B} = (\beta_1, \beta_2, \ldots, \beta_n)$ be an ordered set partition of $[n]$  
    whose (possibly empty) parts satisfy $|\beta_i| = \ell_i + 1$ for all $i \in [n]$.  
    A \defterm{labeled \Luka~path} is the pair $L = (\ell, \mathcal{B})$ obtained by labeling  
    the portion of the $x$-axis under step $\ell_i$ with the elements of $\beta_i$,  
    listed in increasing order.   If $\ell_i + 1 = 0$, then step $\ell_i$ is vertical, so $\beta_i$ is empty and no labeling occurs.

    We let $\alpha_L$ denote the permutation obtained by ordering the elements of each $\beta_j$ in increasing order for $j \in [n]$. This permutation $\alpha_L$ is useful in describing the descents of parking functions in terms of $L$.
\end{definition}

  \begin{remark} 
  Each labeled \Luka~path $L = (\ell, \mathcal{B})$ determines a parking function $\alpha$ by setting  
   $\alpha_k = m$ whenever $k \in \beta_m$.  
   Equivalently, the interval on the $x$-axis under step $\ell_m$ is labeled by the elements of $\beta_m$,  
   which may be interpreted as cars, with their block $\beta_m$ indicating their preference. 

  This process can also be reversed. As discussed in \Cref{def:luk_word}, each parking function determines a \L{}ukasiewiczpath $\ell$. For the corresponding labeled \L{}ukasiewiczpath, the set partition $\mathcal{B}$ is obtained by labeling, along the $x$-axis under each step $\ell_i$, the intervals with the cars that prefer spot $i$, listed in increasing order. 
\end{remark}

We illustrate these concepts with an example and show how to obtain a parking function from a labeled \Luka~path.

\begin{example}\label{ex:Luka to pf}
Consider the \Luka~path  \(\ell = (2,0,1,0,-1,0,-1,-1)\) illustrated in \Cref{fig:Lukapath2}, together with the ordered set partition  $\mathcal{B}= (\beta_1,\beta_2,\ldots, \beta_n)$ where
\[
\beta_1 = \{2,4,6\}, \ \beta_2 = \{1\}, \ \beta_3 = \{3,5\}, \ \beta_4 = \{8\}, \ \beta_5 = \emptyset, \ \beta_6 = \{7\}, \ \beta_7 = \emptyset, \ \beta_8 = \emptyset.
\]  

For each step $\ell_i$, the elements of $\beta_i$ are placed, in increasing order, along the intervals of the $x$-axis corresponding to that step, as shown in \Cref{fig:Lukapath2}.  

    \begin{figure}[ht]
    \centering
    \begin{tikzpicture}
           \draw[step=1cm,gray,very thin] (0,0) grid (8,3);
           \foreach \x in {0,...,3} {
        \draw (0.1,\x) -- (-0.1,\x) node[left] {\x};
    }
    \foreach \x in {0,...,8} {
    \draw (\x,0.1) -- (\x,-0.1) node[below] {\x};
    }
    \draw[ultra thick, red](0,0)--(3,2)--(4,2)--(6,3)--(7,3)--(7,2)--(8,2)--(8,1)--(8,0);
    \foreach \x/\y in {0/0, 3/2, 4/2, 6/3,7/3, 7/2,8/2, 8/1,8/0} {
        \filldraw[red] (\x,\y) circle (3pt);
    }
    \node at (1.5,1.5) {$\ell_1$};
    \node at (3.5,2.4) {$\ell_2$};
    \node at (5,2.9) {$\ell_3$};
    \node at (6.5,3.3) {$\ell_4$};
    \node at (6.8,2.5) {$\ell_5$};
    \node at (7.5,2.3) {$\ell_6$};
    \node at (8.3,1.5) {$\ell_7$};
    \node at (8.3,.5) {$\ell_8$};

\node at(.5,-.18){\tb{2}};
\node at(1.5,-.18){\tb{4}};
\node at(2.5,-.18){\tb{6}};
\node at(3.5,-.18){\tb{1}};
\node at(4.5,-.18){\tb{3}};
\node at(5.5,-.18){\tb{5}};
\node at(6.5,-.18){\tb{8}};
\node at(7.5,-.18){\tb{7}};

    \end{tikzpicture}
    \caption{The labeled \Luka~path $L=(\ell,\beta)$ described in \Cref{ex:Luka to pf}.}
    \label{fig:Lukapath2}
\end{figure}
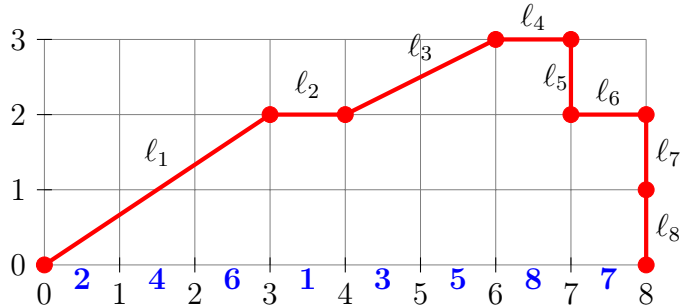

The corresponding parking function is  \(\alpha = (2,1,3,1,3,1,6,4)\),
obtained by setting $\alpha_k = m$ whenever $k \in \beta_m$.  It is worth noting that $\beta_i$ is the list of cars preferring spot $i$ in $\alpha$.   Finally, the permutation $\alpha_L$ associated to $L$ from \Cref{def:labeled_Luka} is  \( \alpha_L = (2,4,6,1,3,5,8,7)\).
\end{example}

As the previous discussion has hinted, there is a one-to-one correspondence between labeled \L{}ukasiewicz~paths and parking functions. 
\begin{theorem}{\cite{SZ}}
\label{thm:labeled luka bijection to PFs}
    The set of labeled \Luka~paths of length $n$ is in bijection with parking functions of length $n$.
\end{theorem}

Selig and Zhu give a direct bijection between non-decreasing parking functions and unlabeled \Luka~paths \cite[Theorem 2.1]{SZ}, indicating the extension to general parking functions in a remark.
Instead of constructing this bijection directly, we prove \Cref{thm:labeled luka bijection to PFs} by first building a bijection between labeled \Luka~paths and labeled Dyck paths.
Since labeled Dyck paths are already known to be in bijection with parking functions \cite{Schumacher}, this will establish the theorem.
We take this route to emphasize the close relationship between \Luka~paths and Dyck paths.

Recall that a Dyck path of length $n$ is a lattice path from $(0,0)$ to $(n,n)$ with north steps ($\North$) and east steps ($\East$) that stays weakly above the diagonal $y=x$.
A labeled Dyck path of length $n$ assigns the labels $[n]$ bijectively to the north steps, with each vertical run labeled in increasing order from bottom to top.
Given such a labeled Dyck path, the associated parking function $\alpha=(\alpha_1,\ldots,\alpha_n)$ is defined by setting $\alpha_i=j$ whenever label $i$ appears on a north step in column $j$ (columns are indexed from $1$).
For further details on this bijection, see \cite{Schumacher}.
We illustrate this correspondence in the following example.

\begin{example}
  Let $\alpha$ be the parking function corresponding to the labeled Dyck path given in \Cref{fig:Dyck}.
Reading by columns: the north-step labels in column~1 are $2,4,$ and $6$, so $\alpha_2=\alpha_4=\alpha_6=1$; 
column~2 has label $1$, so $\alpha_1=2$; 
column~3 has labels $3,5$, so $\alpha_3=\alpha_5=3$; 
column~4 has label $8$, so $\alpha_8=4$; 
and column~6 has label $7$, so $\alpha_7=6$.
Thus \(\alpha=(2,1,3,1,3,1,6,4).\)

  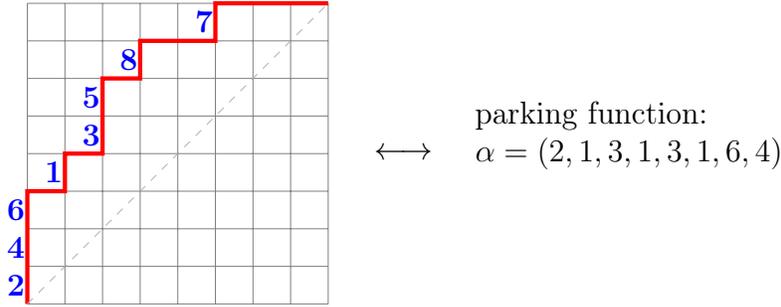
\begin{figure}[H]
        \centering
        \begin{tikzpicture}[scale=.5]
            \draw[step=1cm,gray,very thin] (0,0) grid (8,8);
    \draw[ultra thick, red](0,0)--(0,3)--(1,3)--(1,4)--(2,4)--(2,6)--(3,6)--(3,7)--(5,7)--(5,8)--(8,8);
\node at(-.3,.5){\tb{2}};
\node at(-.3,1.5){\tb{4}};
\node at(-.3,2.5){\tb{6}};
\node at(.7,3.5){\tb{1}};
\node at(1.7,4.5){\tb{3}};
\node at(1.7,5.5){\tb{5}};
\node at(2.7,6.5){\tb{8}};
\node at(4.7,7.5){\tb{7}};
\draw[gray!60,dashed](0,0)--(8,8);
\node at (15,5){parking function:};
\node at (16,4){$\alpha=(2,1,3,1,3,1,6,4)$};
\node at (10,4){$\longleftrightarrow$};
    \end{tikzpicture}
        \caption{A labeled Dyck path of length $8$ and its corresponding parking function.}
        \label{fig:Dyck}
    \end{figure}
For the reverse direction, given the parking function $\alpha=(2,1,3,1,3,1,6,4)$,  
we first group cars by preference.  
Notice that these sets form $\mathcal{B}$, which is part of the data of the labeled \Luka~path $L=(\ell,\mathcal{B})$ corresponding to $\alpha$ where
\[
\beta_1=\{2,4,6\},
\beta_2=\{1\},
\beta_3=\{3,5\},
\beta_4=\{8\},
\beta_5=\emptyset,
\beta_6=\{7\},
\beta_7=\emptyset,
\beta_8=\emptyset.
\]
  
The column heights of the corresponding Dyck path are determined by the sizes of these sets:
\(c = (3,1,2,1,0,1,0,0)\). Note that  
$c = \ell + (1,\ldots,1)$.  
We then construct the Dyck path by taking, for each column $j$, $c_j$ north steps followed by one east step, which yields
\(
\North\,\North\,\North\,\East\;
\North\,\East\;
\North\,\North\,\East\;
\North\,\East\;
\East\;
\North\,\East\;
\East\;
\East.
\)
Finally, label the north steps in each column from bottom to top with the elements of $\beta_j$ in increasing order. This reproduces the labeled Dyck path in \Cref{fig:Dyck}.
\end{example}

With this refresher in place, we are now ready to construct a bijection between labeled Dyck paths and labeled \Luka~paths.

\begin{proof}[Proof of \Cref{thm:labeled luka bijection to PFs}] 
Let $L=(\ell,\mathcal{B})$ be a labeled \Luka~path with  
$\ell = (\ell_1,\ell_2,\ldots,\ell_n)$ and $\mathcal{B} = (\beta_1,\beta_2,\ldots,\beta_n)$.  
We construct the corresponding Dyck path by interpreting $\ell_i+1$ as the height of column~$i$ as follows: for $i=1,2,\ldots,n$, take $\ell_i+1$ north steps ($\North$) followed by one east step ($\East$).

Here $\ell_i+1 = |\beta_i|$ is the number of north steps in column~$i$,  
which equals the number of cars in the corresponding parking function that prefer spot~$i$.  
The fact that $\ell$ never goes below the $x$-axis guarantees that the constructed path  
never goes below the main diagonal, so it is indeed a Dyck path.

This procedure is reversible: given a Dyck path, the column heights recover $\ell$,  
and the labels on the vertical runs recover $\mathcal{B}$.  
Since the labeling of north steps of a Dyck path assigns the cars in increasing order from bottom to top,  each element in $\beta_i$ is in increasing order.   
This completes the bijection.
\end{proof}

In the following example, we illustrate how to move between a labeled \Luka~path and labeled Dyck path, as described in the proof of \Cref{thm:labeled luka bijection to PFs}.

\begin{example}
    Consider the labeled \Luka~path of \Cref{ex:Luka to pf} presented in  \Cref{fig:Lukapath2} (also the left of \Cref{fig:side by side}) where $$L=((2,0,1,0,-1,0,-1,-1),(\{2,4,6\},\{1\},\{3,5\},\{8\},\emptyset,\{7\},\emptyset,\emptyset)).$$ 
    Using the steps given in the proof of \Cref{thm:labeled luka bijection to PFs},  we show that the Dyck path given in \Cref{fig:Dyck} corresponds to $L$. For the reader’s convenience, the two paths are presented side by side in \Cref{fig:side by side}. 

    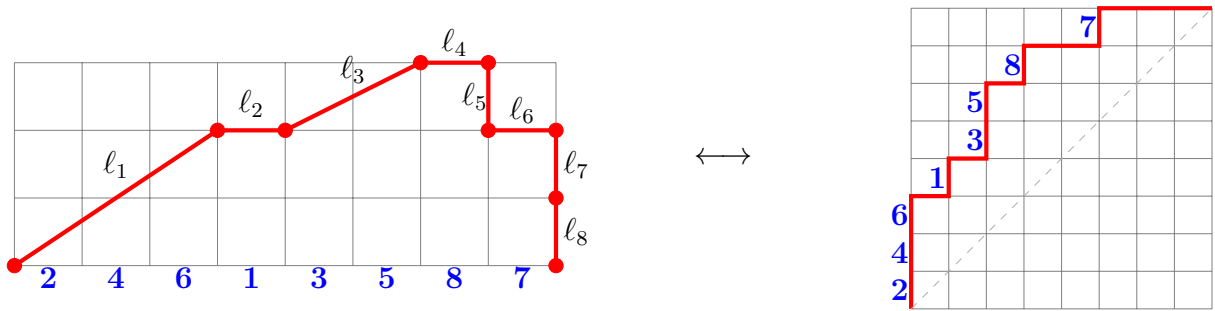
\begin{figure}[ht]
\centering
\begin{minipage}{0.48\textwidth}
    \centering
    \begin{tikzpicture}[scale=.9]
        \draw[step=1cm,gray,very thin] (0,0) grid (8,3);
        \draw[ultra thick, red](0,0)--(3,2)--(4,2)--(6,3)--(7,3)--(7,2)--(8,2)--(8,1)--(8,0);
        \foreach \x/\y in {0/0, 3/2, 4/2, 6/3,7/3, 7/2,8/2, 8/1,8/0} {
            \filldraw[red] (\x,\y) circle (3pt);
        }
        \node at (1.5,1.5) {$\ell_1$};
        \node at (3.5,2.4) {$\ell_2$};
        \node at (5,2.9) {$\ell_3$};
        \node at (6.5,3.3) {$\ell_4$};
        \node at (6.8,2.5) {$\ell_5$};
        \node at (7.5,2.3) {$\ell_6$};
        \node at (8.3,1.5) {$\ell_7$};
        \node at (8.3,.5) {$\ell_8$};
\node at(.5,-.18){\tb{2}};
\node at(1.5,-.18){\tb{4}};
\node at(2.5,-.18){\tb{6}};
\node at(3.5,-.18){\tb{1}};
\node at(4.5,-.18){\tb{3}};
\node at(5.5,-.18){\tb{5}};
\node at(6.5,-.18){\tb{8}};
\node at(7.5,-.18){\tb{7}};
    \end{tikzpicture}
\end{minipage}
\hfill
\begin{minipage}{0.48\textwidth}
    \centering
    \begin{tikzpicture}[scale=.5]
        \draw[step=1cm,gray,very thin] (0,0) grid (8,8);
        \draw[ultra thick, red](0,0)--(0,3)--(1,3)--(1,4)--(2,4)--(2,6)--(3,6)--(3,7)--(5,7)--(5,8)--(8,8);
\node at(-.3,.5){\tb{2}};
\node at(-.3,1.5){\tb{4}};
\node at(-.3,2.5){\tb{6}};
\node at(.7,3.5){\tb{1}};
\node at(1.7,4.5){\tb{3}};
\node at(1.7,5.5){\tb{5}};
\node at(2.7,6.5){\tb{8}};
\node at(4.7,7.5){\tb{7}};
        \draw[gray!60,dashed](0,0)--(8,8);
        \node at (-5,4){$\longleftrightarrow$};
    \end{tikzpicture}
\end{minipage}
\caption{Illustration of the bijection in \Cref{thm:labeled luka bijection to PFs}.}\label{fig:side by side}
\end{figure}

Since $\ell_1 = 2$, the Dyck path begins with three north steps followed by an east step $\North\North\North\East$.  
This is followed by $\North\East$  from $\ell_2 = 0$,  
then $\North\North\East$ from $\ell_3 = 1$.  
From $\ell_4 = 0$ we again have $\North\East$,  
while $\ell_5 = -1$ contributes a single east step $\East$.  
Next, $\ell_6 = 0$ yields $\North\East$,  
$\ell_7 = -1$ contributes another east step $\East$,  
and finally $\ell_8 = -1$ gives the last east step $\East$. Thus the Dyck path is $\North\,\North\,\North\,\East\;
\North\,\East\;
\North\,\North\,\East\;
\North\,\East\;
\East\;
\North\,\East\;
\East\;
\East$.  We now take the ordered set partition $(\{2,4,6\},\{1\},\{3,5\},\{8\},\emptyset,\{7\},\emptyset,\emptyset)$  
and use the elements of the $i$th block to label, in increasing order, the run of north steps in column $i$.  
This produces exactly the labeled Dyck path shown on the right of \Cref{fig:side by side}.
\end{example}

\begin{remark}
The labeled (prime) \Luka~path perspective provides an alternative derivation for the number of (prime) parking functions of length $n$. Let $S_n$ denote the set of permutations of $[n]$. Given any \Luka~path $\ell \in \LP_n$ and any permutation $\pi \in S_n$, labeling the $x$-axis under $\ell$ according to $\pi$ produces a parking function.  
However, different permutations in $S_n$ may yield the same parking function, since permuting the labels within a given step $\ell_i$ does not change the resulting parking function.
 Thus, the total number of parking functions is given by $$|\PF_n|=\sum_{\ell\in\LP_n}\frac{n!}{\prod\limits_{i=1}^{n}(\ell_i+1)!} .$$
Similarly, denoting the set of prime \Luka~paths by $\LP'_n,$ we have
    $$|\PPF_n|=\sum_{\ell\in\LP'_n}\frac{n!}{\prod\limits_{i=1}^{n}(\ell_i+1)!}. $$
\end{remark}

We now define the descent, ascent, and tie set of a labeled \Luka~path.

\begin{definition}\label{def:sets for Luka}
    Let $L=(\ell,\mathcal{B})$ be a labeled \Luka~path of length $n$  with 
    $\ell=(\ell_1,\ell_2,\ldots,\ell_n)$ and $\mathcal{B}=(\beta_1,\beta_2,\ldots,\beta_n)$.
    We define the descent, ascent, and tie set of $L$ as follows:
    \begin{align*}
        \Tie(L)&=\{i\in[n-1]:i,i+1\in\beta_j\mbox{ for some $j$}\},\\
        \Des(L)&=\{i\in[n-1]:\mbox{$i\in \beta_j$ and $i+1\in \beta_m$ with $m<j$}\},\\
        \Asc(L)&=\{i\in[n-1]:\mbox{$i\in \beta_j$ and $i+1\in \beta_m$ with $m>j$}\}.
    \end{align*}
\end{definition}

We illustrate \Cref{def:sets for Luka} in the following example.

\begin{example} 
Consider the labeled \Luka~path 
$$L=((2,0,1,0,-1,0,-1,-1),(\{2,4,6\},\{1\},\{3,5\},\{8\},\emptyset,\{7\},\emptyset,\emptyset)),$$
which is illustrated in \Cref{fig:Lukapath2}. 
The descent, ascent and tie sets of $L$ are as follows:
\begin{align*}
    \Des(L)&=\{1,3,5,7\},\\
    \Asc(L)&=\{2,4,6\},\\
    \Tie(L)&=\emptyset.
\end{align*}

Notice that these sets agree with the  descent, ascent, and tie set of the corresponding parking function $\alpha=(2,1,3,1,3,1,6,4)$ of $L$.
\end{example}

The observation made at the end of this example holds in general: the descent set of a labeled \Luka~path coincides with the descent set of the corresponding parking function. First recall from \Cref{def:labeled_Luka} that, for a labeled \Luka~path $L=(\ell,\mathcal{B})$ of length $n$,  $\alpha_L$ is the permutation arising from ordering the elements of each $\beta_j$ in increasing order for $j \in [n]$.

\begin{proposition}\label{lem:descents agree}
The descent set of the parking function corresponding to a labeled \Luka~path $L=(\ell, \mathcal{B})$ is given by the inverse descent set of $\alpha_L$ (descent set of $\alpha_L^{-1}$). 
\end{proposition}
\begin{proof}
Let $\alpha$ be the parking function corresponding to $L$.  
By definition, $\alpha_k = m$ if and only if $k \in \beta_m$.  
Thus, $\alpha$ has a descent at $i$ exactly when $i \in \beta_j$ and $i+1 \in \beta_m$ for some $m < j$, which is precisely the condition for $L$ to have a descent at $i$.  

Since the elements of each $\beta_j$ are listed in increasing order, the condition $i \in \beta_j$ and $i+1 \in \beta_m$ with $m < j$ is equivalent to $i+1$ appearing before $i$ in $\alpha_L$.  
That is, $\alpha_L^{-1}(i) > \alpha_L^{-1}(i+1)$, meaning $i$ is an inverse descent of $\alpha_L$.  
\end{proof}

\begin{remark}
As with descent sets, the ascent set of a labeled \Luka~path and its corresponding parking function coincide. However, while in \Cref{lem:descents agree} we recover descents via $\alpha_L^{-1}$, both ascents and ties in $\alpha_L$ contribute to inverse ascents:
\[
\Asc(\alpha_L^{-1})=\Asc(\alpha)\,\cup\,\Tie(\alpha).
\] 
Take, for example, $\alpha_L=(1,2,6,4,3,5,8,7).$ Here $\Des(L)=\{3,5,7\},$ $\Asc(L)=\{2,4,6\},$ and $\Tie(L)=\{1\},$ while ascent set of $\alpha_L^{-1}=\{1, 2, 4, 6\}.$
\end{remark}

We conclude this section by noting that the same results hold for prime parking functions via prime \Luka~paths.

\section{$\ell$-forward differences} \label{sec:l-forward diff}
Denote the set of prime parking functions of length $n$ by $\PPF_n$. We use the circular rotation construction, due to Kalikow \cite[pp.~141-142]{Stanley}, to choose a prime parking function uniformly at random.

\begin{enumerate}
\item Pick at random an element $\pi_0 \in (\mathbb{Z}/(n-1)\mathbb{Z})^n$, where the equivalence class representatives are taken in $1,2, \dots, n-1$.
\item Let $(1,1, \dots,
  1)$ be a vector of length $n$.  There is exactly one value of $i \in \{0,1, \dots, n-2\}$, such that $\pi=\pi_0+i(1,1, \dots,
  1)$ (modulo $n-1$) is a prime parking function. Here and throughout this construction, addition is performed coordinate-wise modulo $n-1$, using representatives in $\{1,2,\ldots,n-1\}$.
\end{enumerate}
Thus, a random prime parking function can be generated by assigning $n$ cars' preferences independently on a
circle of length $n-1$ and then applying circular rotation.  As an example, for $n=4$ take $\pi_0=(2,3,3,2).$  The unique parking function of the form $\pi=\pi_0+i(1,1,1,1)$ is given by $(1,2,2,1)=(2,3,3,2)+2(1,1,1,1) \bmod{3}.$ For applications of this circular rotation perspective on parking functions, see Stanley and Yin \cite{StanleyYin}.

Given a tuple $\x=(x_1,x_2,\ldots,x_n)$, we say $\x$ has an {\bf $\ell$-forward difference} at $i$ if $x_{i+1}-x_{i}=\ell\pmod{n-1}$, for $i=1,2,\ldots, n-1$ and $\ell=0,1,\ldots, n-2$.  
Note, the $0$-forward differences are the ties of the parking function.  
We denote the number of $\ell$-forward differences in a parking function $\pi$ by $\fd(\pi).$ 

\PPFRotationCount*

\begin{proof}
We employ the circular rotation construction described above. For $1\leq i\leq n-1$, note that if the
$(i+1)$th car prefers the spot that is (clockwise) $\ell$-distance away from that of the preference of the $i$th car, which happens with probability $1/(n-1)$, then the count for $\ell$-distance adds 1. Otherwise, with probability $(n-2)/(n-1)$, the count stays the same.

Since the number of prime parking functions of length $n$ is
$(n-1)^{n-1}$, it follows that the number of prime parking functions of length $n$ with exactly $k$ $\ell$-forward differences is $a_k=(n-1)^{n-1}\binom{n-1}{k}\left(\frac{1}{n-1}\right)^k\left(\frac{n-2}{n-1}\right)^{n-1-k}.$  Thus, 
\begin{align*}
\sum_{\pi \in \PPF_n} q^{\fd(\pi)}&=\sum_{k=0}^{n-1}a_kq^k=(n-1)^{n-1} \left(q\frac{1}{n-1}+\frac{n-2}{n-1}\right)^{n-1} 
=(q+n-2)^{n-1}.\qedhere
\end{align*}
\end{proof}

Using this generating function for the total number of $\ell$-forward differences in prime parking functions of length $n$, we can calculate the expected number of descents $\E[\des(\pi)|\pi\in\PPF_n].$ 

\begin{corollary}
    Take any 
    integer $\ell \in [0, n-2]$. Then we have $$\E[\fd(\pi)|\pi\in \PPF_n]=1.$$ Moreover, $$\E[{\des}(\pi)|\pi\in\PPF_n]=\E[\asc(\pi)|\pi\in\PPF_n]=\frac{n-2}{2}.$$
\end{corollary}
\begin{proof}
  Recall that  $$\E[\fd(\pi)|\pi\in \PPF_n]= \frac{\sum_{\pi \in \PPF_n} \fd(\pi)}{|\PPF_n|}.$$
  We can use \Cref{prop:q_l-des} to calculate the expected number of $\ell$-forward differences. Differentiating both sides of
    \begin{equation*}
        \sum_{\pi\in\PPF_n}q^{\fd(\pi)}=(q+n-2)^{n-1},
    \end{equation*} we have $$\sum_{\pi\in\PPF_n}\fd(\pi)q^{(\fd(\pi)) -1}=(n-1)(q+n-2)^{n-2}.$$ Evaluating at $q=1,$ 
    $$\sum_{\pi\in\PPF_n}\fd(\pi)=(n-1)^{n-1}.$$  
    Dividing by the total number of prime parking functions, we find $\E[\fd(\pi)|\pi\in \PPF_n]=1$. 

Given $\pi=(\pi_1,\pi_2,\ldots,\pi_n)\in \PPF_n$, each $(\pi_i,\pi_{i+1})$ pair corresponds to an $\ell$-forward difference for some value of $\ell$.  When $\ell=0$, we have $(\pi_i,\pi_{i+1})$ is a tie.  When $\ell\not=0$, $(\pi_i,\pi_{i+1})$ is either an ascent or a descent.  Let $\hat{\pi}\in \PPF_n$ be the reverse parking function, $\hat{\pi}=(\pi_n,\ldots, \pi_2,\pi_1).$  Each ascent of $\pi$ corresponds to a descent of $\hat{\pi}$, and similarly, each descent to an ascent.  Note that in the case $\pi=\hat{\pi},$ $\des(\pi)=\asc(\pi).$  Thus $$\sum_{\ell\not=0}\sum_{\pi\in\PPF_n}\fd(\pi)=\sum_{\pi\in\PPF_n}\des(\pi)+\sum_{\pi\in\PPF_n}\asc(\pi)=2\sum_{\pi\in\PPF_n}\des(\pi),$$ and \[\E[\des(\pi)|\pi\in\PPF_n]=\frac{\sum_{\pi\in\PPF_n}\des(\pi)}{|\PPF_n|}=\frac{1}{2}\frac{\sum_{\ell=1}^{n-2}(n-1)^{n-1}}{(n-1)^{n-1}}=\frac{n-2}{2}.\qedhere\]        
\end{proof}

\section{Tie sets and quasisymmetric functions}\label{sec: TieSets QuasiSymm}

In this section, we prove a link between sums of Gessel's quasisymmetric functions indexed by tie sets and sums of Schur functions.  The corresponding results, in the context of (classical) parking functions and limited to ties and $1$-forward differences, are derived in forthcoming work by Celano, \cite{Celano:SchurPositivity}.

Let $\SetL{\ell}(\pi)$ denote the set of indices where the prime parking function $\pi$ has an $\ell$-forward difference.  Thus $\SetL{0}(\pi)=\Tie(\pi).$  We start with the following count of prime parking functions with a given $\ell$-forward difference set, which is used in proving the main result of this section.

\begin{theorem}\label{thm:Tie_Set_count}
    Given a set $S\subseteq[n-1]$, and fixed $0\leq \ell\leq n-2,$ the number of prime parking functions $\pi\in\PPF_n$ with $\SetL{\ell}(\pi)=S$ is  
    $$|\{\pi \in \PPF_n : \SetL{\ell}(\pi)=S\}| =(n-2)^{n-1-|S|}.$$
\end{theorem}

\begin{proof} 
Consider the mapping $\mL:\PPF_n\rightarrow \left(\mathbb{Z}/(n-1)\mathbb{Z}\right)^{n-1}$ given by 
\[\mL(\pi)=(\mL_1,\mL_2,\ldots,\mL_{n-1})=(\pi_2-\pi_1, \pi_3-\pi_2,\ldots,\pi_n-\pi_{n-1})\pmod{n-1},\] i.e., $\mL_i$ is the $\ell$-forward difference of $\pi$ at $i.$ A consequence of Kalikow's construction is that this mapping is a bijection.  
Given an arbitrary $\mathcal{L}\in \left(\mathbb{Z}/(n-1)\mathbb{Z}\right)^{n-1},$ let \[\pi_0=(1,1+\mL_1,1+\mL_1+\mL_2,\ldots,1+\sum_{j=1}^{n-1}\mL_j)\pmod{n-1}.\]  The prime parking function $\pi$ of the form $\pi_0+i(1,1,\ldots,1)\pmod{n-1}$ is the unique element of $\PPF_n$ with $\mathcal{L}(\pi)=\mL.$

Given a set $S\subseteq [n-1],$ the set of $\pi\in\PPF_n$ with $\SetL{\ell}(\pi)=S$ is precisely the set of prime parking functions that satisfy $\mL(\pi)_i=\ell$ for all $i\in S$ and $\mL(\pi)_i\not=\ell$ for all $i\not\in S.$  As there are exactly $(n-2)^{n-1-|S|}$ such $\mL\in \left(\mathbb{Z}/(n-1)\mathbb{Z}\right)^{n-1},$ it follows that \[|\{\pi \in \PPF_n : \SetL{\ell}(\pi)=S\}| =(n-2)^{n-1-|S|}.\qedhere\]
\end{proof}

We are now ready to prove the following result on the sum, over all prime parking functions of length $n,$ of quasisymmetric functions indexed by tie sets. 

\begin{theorem}\label{thm:quasisym}
For $n\geq 1$,
\begin{equation*}
\sum_{\pi\in\PPF_n}F_{n,\Tie(\pi)} =\sum_{i=1}^n (n-2)^{i-1}s_{(i,1^{n-i})}.
\end{equation*}
\end{theorem}
\noindent Here $F_{n,S}$ is the Gessel's quasisymmetric function defined as
\[F_{n,S}=
\sum_{\substack{1\leq b_1\leq \cdots\leq b_n \\ \\i\in S\Rightarrow b_i<b_{i+1}}} x_{b_1}\cdots x_{b_n},\]
where $S\subseteq[n-1],$ and $s_{(i,1^{n-i})}$ is the Schur function corresponding to the partition $(i,1^{n-i})$. 

We first illustrate this result with an example and then present the proof.

\begin{example}
When $n=3$, $\PPF_3=\{111,112,121,211\}$ and we have 
\begin{align}
    &\Tie(111)=\{1,2\}\rightarrow F_{3,\{1,2\}}(x_1, x_2, x_3)=x_1x_2x_3\label{tie12}\\
     &\Tie(121)=\emptyset\rightarrow \notag \\ &\hspace{1cm} F_{3,\emptyset}(x_1, x_2, x_3)=x_1^3+x_2^3+x_3^3+x_1^2x_2+x_1x_2^2+x_1^2x_3+x_1x_3^2+x_2^2x_3+x_2x_3^2+x_1x_2x_3\label{tiempty}\\
    &\Tie(112)=\{1\}\rightarrow F_{3,\{1\}}(x_1, x_2, x_3)=x_1x_2^2+x_1x_3^2+x_2x_3^2+x_1x_2x_3\label{tie1}\\
    &\Tie(211)=\{2\}\rightarrow F_{3,\{2\}}(x_1, x_2, x_3)=x_1^2x_2+x_1^2x_3+x_2^2x_3+x_1x_2x_3.\label{tie2}
\end{align}
Taking the sum of \eqref{tie12}-\eqref{tie2} yields
\begin{multline}
\label{eqLHS}
\sum_{\pi\in\PPF_n}F_{n,\Tie(\pi)} (x_1, x_2, x_3)
\\
\hspace{-2cm}=x_1^3+x_2^3+x_3^3+2(x_1^2x_2+x_1x_2^2+x_1^2x_3+x_1x_3^2+x_2^2x_3+x_2x_3^2)+4x_1x_2x_3.
\end{multline}

Now for the Schur functions side of the equation in \Cref{thm:quasisym}:
\[s_{(1^3)}+s_{(2,1)}+s_{(3)}.\]
We want to construct all of the fillings  which have numbers weakly increasing along the rows and strictly increasing down the columns of the Young diagram.
For $s_{(2,1)}(x_1, x_2, x_3)$ we have the following fillings:
\[\begin{ytableau}
       1 &  1  \\
  2\\
\end{ytableau},~\begin{ytableau}
       1 &  1  \\
  3\\
  \end{ytableau},~\begin{ytableau}
       1 &  2 \\
  2\\
  \end{ytableau},~\begin{ytableau}
       1 &  2  \\
  3\\
\end{ytableau},~\begin{ytableau}
       1 &  3 \\
  2\\
\end{ytableau}
,~
\begin{ytableau}
       1 &  3\\
3\\
\end{ytableau},~
\begin{ytableau}
       2 &  2  \\
  3\\
\end{ytableau},~
\begin{ytableau}
       2 &  3  \\
  3\\
\end{ytableau}.\]
Now for each filling we construct a monomial, and this yields:
\begin{align}\label{s21}
s_{(2,1)}(x_1, x_2, x_3)=x_1^2x_2+x_1^2x_3+x_1x_2^2+2x_1x_2x_3+x_1x_3^2+x_2^2x_3+x_2x_3^2.
\end{align}
Similarly, we can construct
$s_{(1^3)}(x_1, x_2, x_3)$ from the filling
\[\begin{ytableau}
       1  \\
  2\\
  3\\
\end{ytableau}\]
and so 
\begin{align}\label{s111}
s_{(1^3)}(x_1, x_2, x_3)&=x_1x_2x_3.
\end{align}
Lastly, $s_{(3)}(x_1, x_2, x_3)$ arises from the Young tableau: 
\[\begin{ytableau}
       1 &  1 &1
\end{ytableau}, ~\begin{ytableau}
       1 &  1 &2
\end{ytableau},~
\begin{ytableau}
       1 &  1 &3
\end{ytableau},~\begin{ytableau}
       1 &  2 &2
\end{ytableau}, ~\begin{ytableau}
       1 &  2 &3
\end{ytableau},\]
\[\begin{ytableau}
       1 &  3 &3
\end{ytableau}, ~\begin{ytableau}
       2 &  2 &2
\end{ytableau}, ~\begin{ytableau}
       2 &  2 &3
\end{ytableau}, ~\begin{ytableau}
       2 &  3 &3
\end{ytableau},~
\begin{ytableau}
       3 &  3 &3
\end{ytableau}.\]
Hence 
\begin{align}\label{s3}
s_{(3)}(x_1, x_2, x_3)&=x_1^3+x_1^2x_2+x_1^2x_3+x_1x_2^2+x_1x_2x_3+x_1x_3^2+x_2^3+x_2^2x_3+x_2x_3^2+x_3^3.\end{align}
Summing \eqref{s21}, \eqref{s111}, and \eqref{s3} yields the right-hand side of the equation in \Cref{thm:quasisym}:
\begin{multline}
    \label{eqRHS}
    \sum_{i=1}^{3}(3-2)^{i-1}s_{(i,1^{n-i})}(x_1, x_2, x_3) \\ =x_1^3+x_2^3+x_3^3+2(x_1^2x_2+x_1x_2^2+x_1^2x_3+x_1x_3^2+x_2^2x_3+x_2x_3^2)+4x_1x_2x_3.
\end{multline}
    Now note that the sum in \eqref{eqRHS} agrees with the sum in \eqref{eqLHS}. This confirms \Cref{thm:quasisym} for $n=3$.
\end{example}

\begin{proof}[Proof of \Cref{thm:quasisym}]
    The expansion of a Schur function into fundamental quasisymmetric functions is a
well-known consequence of Stanley's theory of $P$-partitions \cite[Theorem 7.19.7, p.~397]{Stanley}.
    It is known that
$$s_{(i,1^{n-i})}=\sum_{\substack{S\subseteq [n-1], |S|=n-i}}F_{n,S}.$$
So, we have 
\begin{equation}
\label{eq:schur}
\sum_{i=1}^n (n-2)^{i-1}s_{i,1^{n-i}}= \sum_{i=1}^n (n-2)^{i-1} \sum_{S\subseteq [n-1], |S|=n-i}F_{n,S}.
\end{equation}
By change of variables $j=n-i$, we can rewrite \eqref{eq:schur} as follows:
\begin{align*}
    \sum_{j=0}^{n-1} (n-2)^{n-j-1} \sum_{S\subseteq [n-1], |S|=j} F_{n,S}
    =&\sum_{j=0}^{n-1}  \sum_{S\subseteq [n-1], |S|=j}  F_{n,S} ~|\pi \in \PPF_n : \Tie(\pi)=S|\\
    =&\sum_{\pi \in \PPF_n} F_{n,\Tie(\pi)},
\end{align*}
where the first equality follows by \Cref{thm:Tie_Set_count}.
\end{proof}

The following results are generalizations of \Cref{prop:q_l-des}, \Cref{{thm:Tie_Set_count}}, and \Cref{thm:quasisym}. 

\begin{proposition}
\label{prop:qq}
Let $m \neq \ell$. Then we have
 $$ \sum_{\pi \in \PPF_n} q^{
\fd(\pi)} t^{\fdm(\pi)} = (q+t+n-3)^{n-1}.$$
\end{proposition}

\begin{proof}
The proof applies the circular rotation idea as in the proof of \Cref{prop:q_l-des}. We note that the number of prime parking functions of length $n$ with exactly $j$ $\ell$-forward differences and $k$ $m$-forward differences is $a_{j,k}=(n-1)^{n-1}\binom{n-1}{j, k}\left(\frac{1}{n-1}\right)^{j+k}\left(\frac{n-3}{n-1}\right)^{n-1-j-k}.$  Thus, 
\begin{align*}
\sum_{\pi \in \PPF_n} q^{
\fd(\pi)} t^{\fdm(\pi)} &=\sum_{j=0}^{n-1} \sum_{k=0}^{n-1-j} a_{j, k} q^j t^k=(n-1)^{n-1} \left(q\frac{1}{n-1}+t\frac{1}{n-1}+\frac{n-3}{n-1}\right)^{n-1} \\
&=(q+t+n-3)^{n-1}.\qedhere
\end{align*}
\end{proof}

\begin{theorem}\label{thm:Tie_Set_count_mixed}
    Let $m \neq \ell$. Given sets $S\subseteq[n-1]$ and $T\subseteq[n-1]$, and fixed $0\leq m, \ell\leq n-2,$ the number of prime parking functions $\pi\in\PPF_n$ with $\SetL{m}(\pi)=S$ and $\SetL{\ell}(\pi)=T$ is  
    $$|\{\pi \in \PPF_n : \SetL{m}(\pi)=S, \SetL{\ell}(\pi)=T\}| =(n-3)^{n-1-|S|-|T|}.$$
\end{theorem}

\begin{proof}
The proof follows a similar line of reasoning as in the proof of \Cref{{thm:Tie_Set_count}}. We note that the set of $\pi\in\PPF_n$ with $\SetL{m}(\pi)=S$ and $\SetL{\ell}(\pi)=T$ is precisely the set of prime parking functions that satisfy $\mL(\pi)_i=m$ for all $i\in S$, $\mL(\pi)_j=\ell$ for all $j\in T$, and $\mL(\pi)_k\not=m,\ell$ for all $k\not\in S \cup T.$
\end{proof}

\PPFQuasiSymmetric*

\begin{proof}
We proceed as in the proof of \Cref{thm:quasisym}. Observe that
\begin{equation}
\label{eq:schur_mixed1}
\sum_{i=1}^n q^{n-i}(n-2)^{i-1} s_{(i,1^{n-i})}= \sum_{i=1}^n q^{n-i}(n-2)^{i-1} \sum_{S\subseteq [n-1], |S|=n-i}F_{n,S}.
\end{equation}
By change of variables $j=n-i$, we can rewrite \eqref{eq:schur_mixed1} as follows:
\begin{align*}
    \sum_{j=0}^{n-1} q^j (n-2)^{n-j-1} \sum_{S\subseteq [n-1], |S|=j} F_{n,S}
    =&\sum_{j=0}^{n-1}  \sum_{S\subseteq [n-1], |S|=j}  q^{|S|} F_{n,S} ~|\pi \in \PPF_n : \SetL{\ell}(\pi)=S|\\
    =&\sum_{\pi \in \PPF_n} q^{\fd(\pi)} F_{n,\SetL{\ell}(\pi)},
\end{align*}
where the first equality follows by \Cref{thm:Tie_Set_count}.

Similarly,
\begin{equation}
\label{eq:schur_mixed2}
\sum_{i=1}^n (q+n-3)^{i-1}s_{i,1^{n-i}}= \sum_{i=1}^n (q+n-3)^{i-1} \sum_{S\subseteq [n-1], |S|=n-i}F_{n,S}.
\end{equation}
By change of variables $j=n-i$, we can rewrite \eqref{eq:schur_mixed2} as follows:
\begin{align*}
    &\sum_{j=0}^{n-1} (q+n-3)^{n-j-1} \sum_{S\subseteq [n-1], |S|=j} F_{n,S}\\
    =&\sum_{j=0}^{n-1} \sum_{k=0}^{n-j-1} \binom{n-j-1}{k} q^k (n-3)^{n-j-k-1} \sum_{S\subseteq [n-1], |S|=j}  F_{n,S}\\
    =&\sum_{j=0}^{n-1} \sum_{k=0}^{n-j-1}   \sum_{\substack{S\subseteq [n-1] \\ |S|=j}} \sum_{\substack{T \subseteq [n-1]\setminus S \\ |T|=k}} q^{|T|} F_{n,S} ~|\pi \in \PPF_n : \SetL{m}(\pi)=S, \SetL{\ell}(\pi)=T| \\
    =&\sum_{\pi \in \PPF_n} q^{
\fd(\pi)} F_{n, \SetL{m}(\pi)},
\end{align*}
where the first equality uses the binomial expansion and the second equality follows by \Cref{thm:Tie_Set_count_mixed}.
\end{proof}

We note that, although the circular rotation and \Luka~path models provide useful perspectives on prime parking functions, neither immediately simplifies the computation of $\E[\pi_1]$. The circular rotation construction is particularly effective for difference- or distance-type statistics since these are unchanged by adding the same value to every preference. However, the value of $\pi_1$ changes under circular rotation. Thus, computing its distribution requires determining which rotation produces the prime parking function. Similarly, although the \Luka~path model allows us to interpret displacement geometrically as area, different paths represent different numbers of parking functions. Computing the expected displacement, and by extension $\E[\pi_1],$ still requires weighing each path by the number of corresponding labels.

\section{Conclusions and future work}\label{sec: conclusions}
Since the inception of the parking problem, researchers have found deep connections between parking
functions and many other combinatorial structures, leading to applications in probability and statistics, algebraic geometry, interpolation theory, and representation theory. As pointed to in the introduction, the study of parking functions has found wide applications in data science through the alternative formulation of the parking problem as a hashing problem. Via this angle, the average displacement that we derived in prime parking functions provides insight into the efficiency of storing and retrieving data in the system. We study parking functions from multiple lenses, concentrating in particular on a subclass of parking functions termed prime parking functions. We expect that similar techniques may be extended to other subclasses of parking functions such as Stirling parking functions and tiered parking functions. In the last section, we display the intriguing link between parking functions and (quasi)symmetric functions. We hope to explore the interconnection between parking functions
and other related topics in combinatorics in future work.

\section*{Acknowledgements}
The authors benefited from participation in the Collaborative Workshop in Algebraic Combinatorics at the Institute for Advanced Study in June 2025.
P.~E.~Harris and M.~Yin were supported in part by an award from the Simons Travel Support for Mathematicians program. S.~Kara was supported by NSF Grant DMS-2418805.

\textbf{Data availability statement.} Data sharing not applicable to this article as no datasets were generated or analysed during the current study.

\bibliographystyle{plain}
\bibliography{bibliography}

\end{document}